\documentclass[preprint,12pt]{elsarticle}

\usepackage[utf8]{inputenc}
\usepackage[T1]{fontenc}
\usepackage{lmodern}
\usepackage{microtype}
\usepackage{amsmath,amssymb,amsthm,mathtools}
\usepackage{enumitem}
\usepackage{hyperref}
\usepackage[nameinlink,capitalize,noabbrev]{cleveref}

\journal{Journal of Mathematical Analysis and Applications}

\numberwithin{equation}{section}

\makeatletter
\def\ps@pprintTitle{%
    \let\@oddhead\@empty
    \let\@evenhead\@empty
    \let\@oddfoot\@empty
    \let\@evenfoot\@empty}
\makeatother

\hypersetup{
    colorlinks=true,
    linkcolor=blue,
    citecolor=blue,
    urlcolor=blue,
    pdftitle={Order-Moment Transport and Hankel Determinants in Special-Function Inequalities},
    pdfauthor={Domingos S. P. Salazar},
    pdfsubject={Order-moment transport, Hankel determinants, q-gamma functions, Mills ratio, Tricomi functions, and real-order Turan inequalities},
    pdfkeywords={total positivity, moment transform, Hankel determinant, q-gamma function, Mills ratio, Tricomi function, Coulomb regularization}
}

\newtheorem{theorem}{Theorem}[section]
\newtheorem{lemma}[theorem]{Lemma}
\newtheorem{proposition}[theorem]{Proposition}
\newtheorem{corollary}[theorem]{Corollary}

\theoremstyle{definition}

\newtheorem{remark}[theorem]{Remark}

\newcommand{\R}{\mathbb{R}}
\newcommand{\N}{\mathbb{N}}
\newcommand{\dd}{\,\mathrm{d}}

\newcommand{\STP}{\mathrm{STP}}
\newcommand{\TN}{\mathrm{TN}}

\begin{document}

\begin{frontmatter}

\title{Order-Moment Transport and Hankel Determinants in Special-Function Inequalities}

\author[addr1]{Domingos S. P. Salazar\corref{cor1}}
\address[addr1]{Unidade de Educa\c{c}\~ao a Dist\^ancia e Tecnologia,
Universidade Federal Rural de Pernambuco,
52171-900 Recife, Pernambuco, Brazil}
\cortext[cor1]{Corresponding author.}

\begin{abstract}
Scalar inequalities in an order parameter often arise as the \(2\times2\)
shadow of a stronger Hankel determinant statement.  We record a
moment-representation criterion: positive exponential and Mellin order
representations, together with gamma-normalized completely monotone averages,
generate totally nonnegative Hankel kernels, with strictness controlled by the
support of the representing measure.  The criterion packages the classical total-positivity mechanism as a
recognition calculus for special-function inequalities, turning the order
parameter into a moment exponent after the correct normalization.

The applications include three named determinant lifts.  First, we prove the
positive Jackson \(q\)-gamma Hankel conjecture of
Karp--Vishnyakova--Zhang: for \(0<q<1\), the kernel
\((x,y)\mapsto\Gamma_q(x+y)\) is \(\STP_\infty\).  This is an atomic
Mellin-moment instance of the general criterion; the reciprocal
sign-regularity problem for \(1/\Gamma_q\) is separate and is not addressed
here.  Second, we answer Yang's continuous half-gamma Mills-ratio
log-convexity question and strengthen it to strict total positivity, hence to
all higher Hankel Tur\'an determinants.  Third, we treat Tricomi rays and the
one-dimensional Coulomb regularization as all-minor Hankel determinant
hierarchies.  For the Coulomb regularization, the \(2\times2\) minor gives the
scalar log-convexity question recorded by Baricz--Pog\'any, and the full
theorem supplies the corresponding all-minor strengthening.
\end{abstract}


\end{frontmatter}

\section{Motivation and scope}

A recurring pattern in special-function inequalities is that the source problem is asked at scalar level.  One asks whether a parameter family is log-convex, whether a Tur\'an expression has a fixed sign, or whether an integer-order inequality admits a real-order interpolation.  These questions are natural, but often underspecify the structure.  Once the right order variable and normalization are exposed, the scalar inequality may become the first nontrivial minor of a totally positive kernel.

The guiding principle is that a scalar order inequality can often be recognized as the first visible case of a determinant statement.  Once the order parameter is realized as a moment exponent, total positivity supplies the full hierarchy of Hankel minors.  The mathematical tools are classical: total positivity and sign-regular kernels \cite{Karlin1968,Pinkus2010}, generalized Vandermonde and Chebyshev-system criteria \cite{Karlin1968,Pinkus2010}, Andreief's identity \cite{Andreief1886,Forrester2019}, and Bernstein--Widder representation for completely monotone functions \cite{Widder1946,SchillingSongVondracek2012}.  Here these tools serve as a recognition calculus for special-function inequalities: the real-order parameter becomes a moment exponent.

Gamma normalization supplies the key exponent-extraction step.  It can turn
a shape parameter into a Mellin exponent.  Let \(g\) be completely monotone on
\((0,\infty)\), with Bernstein--Widder representation
\begin{equation}\label{eq:intro-cm-gamma-trigger}
        g(y)=\int_{[0,\infty)}e^{-sy}\,\dd\mu(s),
\end{equation}
where \(\mu\) is a positive measure.  Then gamma averaging gives
\begin{equation}\label{eq:intro-gamma-extraction}
        \frac{1}{\Gamma(a)}\int_0^\infty y^{a-1}e^{-y}g(y)\,\dd y
        =
        \int_{[0,\infty)}(1+s)^{-a}\,\dd\mu(s).
\end{equation}
Thus the parameter \(a\) becomes the exponent in a positive moment
representation.  This elementary identity is the trigger for the
special-function applications below.

We distinguish this recognition step from several existing strands of the
literature.  Moment kernels and generalized Vandermonde determinants are
standard objects in total positivity theory \cite{Karlin1968,Pinkus2010}.
Determinantal inequalities for special functions are also well developed:
Ismail and Laforgia proved monotonicity properties of determinant functions
involving special functions \cite{IsmailLaforgia2007}; Baricz and Ismail
proved sharp Tur\'an inequalities and complete-monotonicity results for Tricomi
Tur\'anians \cite{BariczIsmail2013}; and Karp and Sitnik studied parameter
log-convexity and log-concavity for hypergeometric-like functions
\cite{KarpSitnik2010}.  Karp, Vishnyakova, and Zhang conjectured strict total
positivity of the Jackson \(q\)-gamma Hankel kernel
\((x,y)\mapsto\Gamma_q(x+y)\), while leaving the reciprocal sign-regularity
problem open \cite{KarpVishnyakovaZhang2025}.  More recently, Yang proved shifted-parameter
complete-monotonicity and convexity consequences for Tricomi functions and
formulated the half-gamma Mills-ratio question treated below
\cite{Yang2025Extension}.  The present paper identifies cases where scalar inequalities lift to a full
Hankel total-positivity structure.

The central model is this.  Suppose
\begin{equation}\label{eq:order-moment-model}
        F_x(\theta)=\int_E e^{\theta\phi(t)}\,\dd\mu_x(t),
\end{equation}
where \(\mu_x\) is a positive measure.  After the push-forward \(u=e^{\phi(t)}\), this becomes a Mellin moment
\begin{equation}\label{eq:mellin-model-intro}
        F_x(\theta)=\int_0^\infty u^\theta\,\dd\nu_x(u).
\end{equation}
Consequently the Hankel kernel \((a,b)\mapsto F_x(a+b)\) is totally nonnegative, and is strictly totally positive whenever the representing measure has enough support.  Ordinary log-convexity is then just the \(2\times2\) shadow of the all-order determinant inequalities
\begin{equation}\label{eq:det-hierarchy-intro}
        \det\big[F_x(a_i+b_j)\big]_{i,j=1}^m>0,
        \qquad a_1<\cdots<a_m,\quad b_1<\cdots<b_m .
\end{equation}

The paper develops this recognition principle before the special-function applications.  The first part proves a general transport theorem for positive exponential and Mellin order representations, then isolates the normalizations that preserve strict total positivity: positive row-column factors, affine order changes, push-forwards, atomic support, and gamma averaging.  This normalization calculus explains why certain gamma factors appear naturally in source Tur\'an questions, and why discrete Mellin support is already enough for strictness of every finite minor.

Three special-function outputs are developed from this principle.  The first is
a direct atomic Mellin application.  The standard Jackson \(q\)-gamma expansion
turns \(\Gamma_q(z)\) into the Mellin transform of a positive measure with
infinitely many atoms, and therefore gives
\begin{equation}\label{eq:intro-qgamma-kernel}
        (a,b)\longmapsto \Gamma_q(a+b)
\end{equation}
as an \(\STP_\infty\) kernel for \(0<q<1\).  This resolves the positive
Hankel part of the Karp--Vishnyakova--Zhang conjectural picture; the
reciprocal sign-regularity problem is not addressed here.

The second output is
a gamma-normalized Mills-ratio theorem.  Yang and Tian proved an integer-order
Tur\'an ratio for higher derivatives of the Mills ratio \cite{YangTian2025},
and Yang later formulated the corresponding continuous half-gamma
log-convexity question for
\begin{equation}\label{eq:intro-yang-half-gamma}
        p\longmapsto \frac{R_p(x)}{\Gamma((p+1)/2)} .
\end{equation}
The gamma-normalized completely monotone theorem answers this question and
upgrades it to all Hankel minors.

The third output is a Tricomi/Coulomb determinant theorem.  For \(z>0\) and
\(\delta<1\), the ray
\begin{equation}\label{eq:intro-tricomi-ray}
        (a,b)\longmapsto U(a+b,a+b+\delta,z)
\end{equation}
is \(\STP_\infty\) on \((0,\infty)^2\).  Since the Coulomb regularization satisfies
\begin{equation}\label{eq:intro-coulomb-tricomi}
        V_q(x)=x^{2q+1}U(q+1,q+3/2,x^2),
\end{equation}
the case \(\delta=1/2\) gives
\begin{equation}\label{eq:intro-coulomb-kernel}
        (a,b)\longmapsto V_{a+b-1}(x)
\end{equation}
as an \(\STP_\infty\) kernel.  Its \(2\times2\) minor recovers the
Baricz--Pog\'any scalar log-convexity question, and the theorem gives the full
all-order determinant strengthening.  Yang's shifted-Tricomi results provide
overlapping scalar log-convexity consequences \cite{Yang2025Extension}.

\subsection*{Relation to prior determinant inequalities}
The determinant implications used below are classical consequences of moment
representations, Andreief's identity, and generalized Vandermonde determinants
\cite{Andreief1886,Karlin1968,Pinkus2010,Forrester2019}.  They support the recognition and normalization results for Mills, Tricomi,
and Coulomb order parameters developed below.

\section{Total positivity and order-moment transport}

We first recall the determinant language \cite{Karlin1968,Pinkus2010}.  Let \(X,Y\) be ordered sets.  A kernel \(K:X\times Y\to\R\) is totally nonnegative of order \(m\), written \(\TN_m\), if
\begin{equation}
        \det[K(x_i,y_j)]_{i,j=1}^r\geq0
\end{equation}
for every \(1\leq r\leq m\) and every strictly increasing \(x_1<\cdots<x_r\) in \(X\) and \(y_1<\cdots<y_r\) in \(Y\).  It is strictly totally positive of order \(m\), written \(\STP_m\), if all these determinants are strictly positive.  We write \(\TN_\infty\) and \(\STP_\infty\) when the property holds for all \(m\).

\begin{lemma}[Generalized Vandermonde sign]\label{lem:vandermonde}
Let \(0<t_1<\cdots<t_m\) and \(\alpha_1<\cdots<\alpha_m\).  Then
\begin{equation}\label{eq:gen-vander}
        \det[t_j^{\alpha_i}]_{i,j=1}^m>0.
\end{equation}
\end{lemma}

\begin{proof}
Write \(t_j=e^{u_j}\), with \(u_1<\cdots<u_m\).  The functions \(u\mapsto e^{\alpha_i u}\) form an extended complete Chebyshev system because
\begin{equation}
        \det\left[\frac{\dd^{r-1}}{\dd u^{r-1}}e^{\alpha_i u}\right]_{i,r=1}^m
        =e^{(\alpha_1+\cdots+\alpha_m)u}
        \prod_{1\leq i<j\leq m}(\alpha_j-\alpha_i)>0.
\end{equation}
Their evaluation determinant has the same positive sign on \(u_1<\cdots<u_m\), by the standard extended Chebyshev-system criterion \cite{Karlin1968,Pinkus2010}.
\end{proof}

\begin{theorem}[Exponential-composition transport]\label{thm:exp-transport}
Let \(\rho\) be a positive Borel measure on \(\R\), and let
\begin{equation}
        H(\theta)=\int_{\R}e^{\theta u}\,\dd\rho(u)
\end{equation}
be finite on an interval \(I\).  Let \(X,Y\) be ordered sets, and let
\(\alpha:X\to\R\), \(\beta:Y\to\R\) be strictly monotone maps with the same orientation: either both are strictly increasing or both are strictly decreasing.  Suppose that
\begin{equation}
        \alpha(x)+\beta(y)\in I
\end{equation}
on the subdomain under consideration.  Then
\begin{equation}
        K(x,y)=H(\alpha(x)+\beta(y))
\end{equation}
is \(\TN_\infty\) on that subdomain.  If \(\rho\) has at least \(m\) distinct support points, then \(K\) is \(\STP_m\) on every ordered subdomain on which the displayed integrals are finite.  If \(\rho\) has infinite support, then \(K\) is \(\STP_\infty\).

If \(\alpha\) and \(\beta\) have opposite orientations, the same argument gives a sign-regular kernel rather than a totally positive kernel in the above ordering.
\end{theorem}

\begin{proof}
First suppose that \(\rho\) is finite and compactly supported.  Andreief's identity \cite{Andreief1886,Forrester2019} gives
\begin{equation}
\det[H(\alpha(x_i)+\beta(y_j))]_{i,j=1}^m
=
\frac{1}{m!}\int_{\R^m}
\det[e^{\alpha(x_i)u_k}]_{i,k=1}^m
\det[e^{\beta(y_j)u_k}]_{j,k=1}^m
\prod_{k=1}^m\dd\rho(u_k).
\end{equation}
After sorting \(u_1,\ldots,u_m\), repeated support values give zero alternating determinants.  On every chamber with distinct ordered values, the two determinants have the same sign.  If \(\alpha,\beta\) are both increasing, both signs are positive by the exponential Chebyshev argument used in \cref{lem:vandermonde}.  If both are decreasing, both determinants acquire the same sign \((-1)^{m(m-1)/2}\), and the product is again positive.  This proves nonnegativity.  Strictness follows by choosing \(m\) disjoint
intervals of positive \(\rho\)-mass and localizing the integral to their
product, where the determinant product is strictly positive.

For a general positive Borel measure, restrict first to a fixed finite ordered
subdomain of the moment interval.  Approximate \(\rho\) by its restrictions to
compact sets on which all entries in the relevant determinant are finite.  The
truncated determinants are nonnegative.  Entrywise convergence of the moment
matrix, followed by continuity of the determinant, gives total
nonnegativity.  Strictness follows directly from the same localization argument
applied to the original measure.
\end{proof}

The Mellin and Laplace forms used below are positive-axis specializations of \cref{thm:exp-transport}.  They are stated separately because most applications arrive in Mellin or gamma-normalized form.

\begin{theorem}[Mellin moment kernels]\label{thm:mellin}
Let \(\mu\) be a positive Borel measure on \((0,\infty)\), and define
\begin{equation}\label{eq:M-def}
        M(z)=\int_0^\infty t^z\,\dd\mu(t)
\end{equation}
where the moments are finite on the relevant interval.  Then the Hankel kernel
\begin{equation}
        K(x,y)=M(x+y)
\end{equation}
 is \(\TN_\infty\).  If \(\mu\) has at least \(m\) distinct support points, then
\begin{equation}\label{eq:M-strict}
        \det[M(x_i+y_j)]_{i,j=1}^m>0
\end{equation}
for every \(x_1<\cdots<x_m\), \(y_1<\cdots<y_m\), provided all sums \(x_i+y_j\) lie in the moment domain.  If \(\mu\) has infinite support, then \(K\) is \(\STP_\infty\) on every ordered subdomain where the required moments are finite.
\end{theorem}

\begin{proof}
First suppose that \(\mu\) is finite and compactly supported in \((0,\infty)\).  Andreief's identity gives
\begin{align}\label{eq:andreief-proof}
\det[M(x_i+y_j)]_{i,j=1}^m
&=\frac{1}{m!}\int_{(0,\infty)^m}
    \det[t_k^{x_i}]_{i,k=1}^m
    \det[t_k^{y_j}]_{j,k=1}^m
    \prod_{k=1}^m\dd\mu(t_k).
\end{align}
After sorting the variables, repeated values contribute zero to the alternating determinants, while on every chamber with distinct ordered values both generalized Vandermonde determinants have the same positive sign by \cref{lem:vandermonde}.  This gives nonnegativity.  If \(\mu\) has at least \(m\) support points, choose disjoint intervals \(J_1<\cdots<J_m\) of positive \(\mu\)-mass.  On a sufficiently small product \(J_1\times\cdots\times J_m\), the product of the two generalized Vandermonde determinants is bounded below by a positive constant, so the integral is strictly positive.

For a general positive Borel measure, restrict first to a fixed finite ordered
subdomain of the moment interval and apply the preceding argument to compact
truncations, for instance to \(\mu_N=\mu|_{[1/N,N]}\) when this exhausts the
relevant support.  The truncated determinants are nonnegative.  Entrywise
convergence of the moment matrix follows from monotone convergence of the
entries, and continuity of determinants then gives total nonnegativity.  The
same interval-localization argument gives strictness whenever the original
support contains \(m\) distinct points.
\end{proof}

\begin{corollary}[Order-moment recognition]\label{cor:order-moment}
Let \(E\) be a measurable space, let \(\mu\) be a positive measure on \(E\), and let \(\phi:E\to\R\) be measurable.  Suppose
\begin{equation}\label{eq:order-moment-general}
        F(\theta)=\int_E e^{\theta\phi(t)}\,\dd\mu(t)
\end{equation}
 is finite on an interval \(I\).  Then \((a,b)\mapsto F(a+b)\) is \(\TN_\infty\) on every ordered subdomain with \(a+b\in I\).  If the push-forward of \(\mu\) under \(t\mapsto e^{\phi(t)}\) has at least \(m\) support points, then the kernel is \(\STP_m\); if the push-forward has infinite support, it is \(\STP_\infty\).
\end{corollary}

\begin{proof}
Push \(\mu\) forward by \(u=e^{\phi(t)}\).  Then \(F(\theta)=\int_0^\infty u^\theta\,\dd\nu(u)\), and \cref{thm:mellin} applies.
\end{proof}

\begin{corollary}[Laplace-Hankel total positivity]\label{cor:laplace}
Let
\begin{equation}\label{eq:laplace-model}
        F(a)=\int_{[0,\infty)}e^{-az}\,\dd\nu(z)
\end{equation}
be finite on an interval \(I\).  Then \((a,b)\mapsto F(a+b)\) is \(\TN_\infty\) on every ordered subdomain where \(a+b\in I\).  If \(\nu\) has at least \(m\) distinct support points, then the kernel is \(\STP_m\).  If \(\nu\) has infinite support, the kernel is \(\STP_\infty\).
\end{corollary}

\begin{proof}
This is \cref{thm:exp-transport} with \(u=z\), \(\rho=\nu\), and \(\alpha(a)=\beta(a)=-a\).  The two order maps are both strictly decreasing, so their orientation signs cancel in the Andreief determinant product.  Equivalently, push \(\nu\) forward by \(z\mapsto e^{-z}\in(0,1]\) and apply \cref{cor:order-moment}.
\end{proof}

\begin{remark}[Classical status]
Theorems~\ref{thm:exp-transport} and~\ref{thm:mellin} are standard
consequences of the total-positivity machinery: Andreief's identity reduces
the minors to products of generalized Vandermonde determinants.  Here they make the later normalization arguments self-contained and track
strictness through support size.
\end{remark}

\section{Normalization calculus}

Special-function formulas introduce normalizations, and the recognition step must survive them.  The following elementary rules will be used repeatedly.

\begin{proposition}[Safe transformations]\label{prop:safe-transformations}
Let \(K:X\times Y\to(0,\infty)\) be \(\STP_m\).
\begin{enumerate}[label=\textup{(\roman*)},leftmargin=2em]
\item If \(\alpha:X\to(0,\infty)\) and \(\beta:Y\to(0,\infty)\), then \(\widetilde K(x,y)=\alpha(x)\beta(y)K(x,y)\) is \(\STP_m\).
\item If \(\varphi:\widetilde X\to X\) and \(\psi:\widetilde Y\to Y\) are strictly increasing maps, then \((u,v)\mapsto K(\varphi(u),\psi(v))\) is \(\STP_m\) on \(\widetilde X\times\widetilde Y\).
\item If \(F(\theta)=M(c\theta+d)\), where \(c>0\), and \((a,b)\mapsto M(a+b)\) is \(\STP_m\), then \((a,b)\mapsto F(a+b)\) is \(\STP_m\) after restricting to the common moment domain.
\item Pointwise limits of \(\TN_m\) kernels are \(\TN_m\), whenever all entries converge.  Strictness may be lost in the limit.
\end{enumerate}
\end{proposition}

\begin{proof}
The first assertion multiplies every minor by
\begin{equation}
        \prod_i\alpha(x_i)\prod_j\beta(y_j)>0.
\end{equation}
The second assertion only reindexes increasing nodes.  For the third assertion,
\begin{equation}
        F(a+b)=M(ca+cb+d)=M\big((ca+d/2)+(cb+d/2)\big),
\end{equation}
and the maps \(a\mapsto ca+d/2\), \(b\mapsto cb+d/2\) are strictly increasing when \(c>0\).  The fourth assertion follows from continuity of determinants under entrywise convergence.
\end{proof}

A particularly important normalization is gamma averaging.  It converts a completely monotone multiplier into a Laplace transform in the shape parameter.

\begin{theorem}[Gamma-normalized exponent extraction]\label{thm:gamma-cm}
Let \(g:(0,\infty)\to[0,\infty)\) be completely monotone, and let \(\mu\) be its Bernstein representing measure:
\begin{equation}\label{eq:bernstein}
        g(y)=\int_{[0,\infty)}e^{-sy}\,\dd\mu(s).
\end{equation}
The existence and uniqueness of this representing measure are the Bernstein--Widder theorem \cite{Widder1946,SchillingSongVondracek2012}.
Define the finite moment domain
\begin{equation}\label{eq:I-mu}
        I_\mu=\left\{a>0:\int_{[0,\infty)}(1+s)^{-a}\,\dd\mu(s)<\infty\right\}
\end{equation}
 and, for \(a\in I_\mu\), set
\begin{equation}\label{eq:Fg-def}
        F_g(a)=\frac{1}{\Gamma(a)}\int_0^\infty y^{a-1}e^{-y}g(y)\,\dd y.
\end{equation}
Then
\begin{equation}\label{eq:Fg-laplace}
        F_g(a)=\int_{[0,\infty)}(1+s)^{-a}\,\dd\mu(s)
        =\int_{[0,\infty)}e^{-a\log(1+s)}\,\dd\mu(s).
\end{equation}
Consequently \((a,b)\mapsto F_g(a+b)\) is \(\TN_\infty\) on ordered subdomains where \(a+b\in I_\mu\).  If the push-forward of \(\mu\) under \(s\mapsto(1+s)^{-1}\) has at least \(m\) support points, then the kernel is \(\STP_m\); if the push-forward has infinite support, it is \(\STP_\infty\).
\end{theorem}

\begin{proof}
By Tonelli's theorem,
\begin{align}
F_g(a)
&=\frac{1}{\Gamma(a)}\int_0^\infty y^{a-1}e^{-y}
        \int_{[0,\infty)}e^{-sy}\,\dd\mu(s)\,\dd y \\
&=\int_{[0,\infty)}\frac{1}{\Gamma(a)}\int_0^\infty y^{a-1}e^{-(1+s)y}\,\dd y\,\dd\mu(s) \\
&=\int_{[0,\infty)}(1+s)^{-a}\,\dd\mu(s).
\end{align}
The total-positivity statement follows from \cref{cor:laplace}.
\end{proof}

\begin{corollary}[All Tur\'an determinants]\label{cor:all-turan}
Under the hypotheses of \cref{thm:gamma-cm}, for every \(m\geq1\) for which the representing measure has at least \(m\) distinct pushed-forward support points,
\begin{equation}\label{eq:all-turan}
        \det[F_g(a_i+b_j)]_{i,j=1}^m>0
\end{equation}
whenever \(a_1<\cdots<a_m\), \(b_1<\cdots<b_m\), and all sums lie in \(I_\mu\).  In particular, if the pushed-forward measure has at least two support points, then \(F_g\) is strictly log-convex.  If the pushed-forward measure has exactly one support point, then \(\log F_g\) is affine on its domain.
\end{corollary}

\begin{remark}[Why the gamma factor matters]
The factor \(\Gamma(a)^{-1}\) is structural.  It converts the
gamma density \(y^{a-1}e^{-y}/\Gamma(a)\) into the identity
\begin{equation}\label{eq:gamma-density-laplace}
        \mathbb{E}_{Y\sim\Gamma(a,1)}[e^{-sY}]=(1+s)^{-a}.
\end{equation}
Thus a Laplace variable \(s\) becomes a Mellin base \((1+s)^{-1}\), and the
shape parameter \(a\) becomes the moment exponent.  This is the mechanism used
in the Mills and Coulomb applications.
\end{remark}

\section{The Jackson \(q\)-gamma Hankel kernel}
\label{sec:qgamma}

The Mellin criterion also settles an atomic source problem.  Karp,
Vishnyakova, and Zhang recorded two separate conjectural statements for the
Jackson \(q\)-gamma function: strict total positivity of the positive Hankel
kernel \((x,y)\mapsto\Gamma_q(x+y)\), and strict sign-regularity of the
reciprocal kernel \((x,y)\mapsto1/\Gamma_q(x+y)\)
\cite[Section~3, items~7--8]{KarpVishnyakovaZhang2025}.  The theorem below
proves only the positive Hankel statement.  The proof is a direct
support-size application of \cref{thm:mellin}.

For \(0<q<1\), write
\begin{equation}\label{eq:qpochhammer-def}
        (a;q)_n=\prod_{j=0}^{n-1}(1-aq^j),
        \qquad
        (a;q)_\infty=\prod_{j=0}^{\infty}(1-aq^j).
\end{equation}
The Jackson \(q\)-gamma function is \cite[Section~5.18(ii)]{DLMF}
\begin{equation}\label{eq:qgamma-def}
        \Gamma_q(z)=(1-q)^{1-z}\frac{(q;q)_\infty}{(q^z;q)_\infty},
        \qquad z>0 .
\end{equation}

\begin{theorem}[Jackson \(q\)-gamma Hankel determinants]\label{thm:qgamma}
For each fixed \(0<q<1\), the kernel
\begin{equation}\label{eq:qgamma-kernel}
        K_q(x,y)=\Gamma_q(x+y)
\end{equation}
is \(\STP_\infty\) on \((0,\infty)^2\).  Equivalently, for every \(m\geq1\),
\begin{equation}\label{eq:qgamma-all-minors}
        \det\big[\Gamma_q(x_i+y_j)\big]_{i,j=1}^m>0
\end{equation}
whenever \(0<x_1<\cdots<x_m\) and \(0<y_1<\cdots<y_m\).
\end{theorem}

\begin{proof}
By Euler's \(q\)-binomial theorem \cite[Section~1.3]{GasperRahman2004},
\begin{equation}\label{eq:qbinomial-use}
        \frac{1}{(\xi;q)_\infty}
        =\sum_{n=0}^{\infty}\frac{\xi^n}{(q;q)_n},
        \qquad |\xi|<1 .
\end{equation}
Taking \(\xi=q^z\), which satisfies \(|\xi|<1\) for \(z>0\), gives
\begin{equation}\label{eq:qbinomial-qz}
        \frac{1}{(q^z;q)_\infty}
        =\sum_{n=0}^{\infty}\frac{q^{nz}}{(q;q)_n},
        \qquad z>0 .
\end{equation}
Therefore
\begin{equation}\label{eq:qgamma-mellin-series}
        \Gamma_q(z)
        =(1-q)(q;q)_\infty
        \sum_{n=0}^{\infty}
        \frac{(q^n/(1-q))^z}{(q;q)_n}.
\end{equation}
Equivalently,
\begin{equation}\label{eq:qgamma-mellin-measure}
        \Gamma_q(z)=\int_0^\infty t^z\,\dd\mu_q(t),
\end{equation}
where
\begin{equation}\label{eq:qgamma-measure}
        \mu_q=(1-q)(q;q)_\infty
        \sum_{n=0}^{\infty}\frac{1}{(q;q)_n}\,
        \delta_{q^n/(1-q)} .
\end{equation}
This is a positive atomic measure with infinitely many distinct support
points.  Its moments in \eqref{eq:qgamma-mellin-measure} are finite for every
\(z>0\): indeed, \((q;q)_n\downarrow(q;q)_\infty>0\), so the \(n\)-th summand
in \eqref{eq:qgamma-mellin-series} is \(O(q^{nz})\).  Applying
\cref{thm:mellin} with \(X=Y=(0,\infty)\) proves
\eqref{eq:qgamma-all-minors}, hence \(K_q\in\STP_\infty\).
\end{proof}

\begin{remark}[The reciprocal problem is separate]
The same paper of Karp--Vishnyakova--Zhang also conjectures strict
sign-regularity for the reciprocal kernel
\((x,y)\mapsto1/\Gamma_q(x+y)\).  \Cref{thm:qgamma} proves only the positive
Hankel kernel statement for \(\Gamma_q\).  The reciprocal kernel is not a
positive Mellin moment kernel, so it lies outside the order-moment mechanism
used here.
\end{remark}

\section{Mills ratio and a gamma-normalized question}
\label{sec:mills}

Let
\begin{equation}
        R(x)=\int_0^\infty e^{-t^2/2}e^{-xt}\,\dd t,
        \qquad x\geq0,
\end{equation}
be the standard integral form of the Gaussian Mills ratio.  Mills introduced the ratio in the classical Gaussian-tail setting \cite{Mills1926}; monotonicity and functional inequalities for it have been studied extensively, including by Baricz \cite{Baricz2008Mills}.  Yang and Tian recently proved an increasing Tur\'an-type ratio for higher derivatives of the Mills ratio \cite{YangTian2025}.  Yang then introduced the real-order moments
\begin{equation}\label{eq:mills-order-moments}
        R_p(x)=\int_0^\infty t^p e^{-t^2/2}e^{-xt}\,\dd t,
        \qquad x>0,\quad p>-1,
\end{equation}
and observed that the Yang--Tian integer-order inequality implies log-convexity of
\begin{equation}
        n\longmapsto \frac{R_n(x)}{\Gamma(n/2+1/2)}
        \qquad n\in\N.
\end{equation}
He stated the corresponding continuous assertion as a guess on \((-1,\infty)\) \cite[Remark~11]{Yang2025Extension}.  The gamma-normalized transport theorem gives a positive answer and the full Hankel-minor hierarchy.  This application is direct: the half-gamma normalization is exactly the gamma-shape normalization exposed by the change of variables \(y=t^2/2\).

\begin{theorem}[Mills-ratio Hankel determinants]\label{thm:mills-ratio}
Fix \(x>0\), and define
\begin{equation}
        \mathcal M_x(p)=\frac{R_p(x)}{\Gamma((p+1)/2)},
        \qquad p>-1.
\end{equation}
Then the kernel
\begin{equation}\label{eq:mills-hankel-kernel}
        (a,b)\longmapsto \mathcal M_x(a+b+1)
\end{equation}
is \(\STP_\infty\) on \((-1,\infty)^2\).  Consequently \(p\mapsto\mathcal M_x(p)\) is strictly log-convex on \((-1,\infty)\); in particular, for \(-1<p<q\),
\begin{equation}\label{eq:mills-yang-question}
        \mathcal M_x\left(\frac{p+q}{2}\right)^2
        <
        \mathcal M_x(p)\mathcal M_x(q).
\end{equation}
More generally,
\begin{equation}\label{eq:mills-all-minors}
        \det\left[\mathcal M_x(a_i+b_j+1)\right]_{i,j=1}^m>0
\end{equation}
whenever \(-1<a_1<\cdots<a_m\), \(-1<b_1<\cdots<b_m\), and \(m\geq1\).
\end{theorem}

\begin{proof}
Set
\begin{equation}
        g_x(y)=e^{-x\sqrt{2y}},
        \qquad y>0.
\end{equation}
The Laplace representation
\begin{equation}\label{eq:sqrt-laplace}
        e^{-x\sqrt{2y}}
        =
        \frac{x}{\sqrt{2\pi}}
        \int_0^\infty s^{-3/2}e^{-x^2/(2s)}e^{-sy}\,\dd s
\end{equation}
shows that \(g_x\) is completely monotone for \(x>0\).  With \(p=2a-1\), the change of variables \(y=t^2/2\) gives
\begin{align}\label{eq:mills-gamma-change}
        R_{2a-1}(x)
        &=2^{a-1}\int_0^\infty y^{a-1}e^{-y}g_x(y)\,\dd y,
        \qquad a>0.
\end{align}
Thus, in the notation of \cref{thm:gamma-cm},
\begin{equation}\label{eq:mills-Fg-link}
        F_{g_x}(a)
        =\frac{1}{\Gamma(a)}\int_0^\infty y^{a-1}e^{-y}g_x(y)\,\dd y
        =2^{1-a}\mathcal M_x(2a-1).
\end{equation}
By \cref{thm:gamma-cm}, \((u,v)\mapsto F_{g_x}(u+v)\) is
\(\STP_\infty\), because the representing measure in
\eqref{eq:sqrt-laplace} has infinite support.  For target variables
\(A,B>-1\), put
\begin{equation}\label{eq:mills-target-affine}
        u=\frac{A+1}{2},\qquad v=\frac{B+1}{2}.
\end{equation}
Then
\begin{equation}\label{eq:mills-target-factor}
        F_{g_x}(u+v)
        =
        2^{-(A+B)/2}\mathcal M_x(A+B+1).
\end{equation}
The maps \(A\mapsto(A+1)/2\) and \(B\mapsto(B+1)/2\) are strictly increasing,
and the factor \(2^{(A+B)/2}=2^{A/2}2^{B/2}\) is a positive row-column
factor.  Hence \((A,B)\mapsto\mathcal M_x(A+B+1)\) is \(\STP_\infty\).

The scalar log-convexity inequality \eqref{eq:mills-yang-question} is the
\(2\times2\) minor obtained by taking
\begin{equation}\label{eq:mills-scalar-minor-choice}
        A_1=B_1=\frac{p-1}{2},
        \qquad
        A_2=B_2=\frac{q-1}{2}.
\end{equation}
\end{proof}

\begin{corollary}[Yang's half-gamma question, all-minor form]
\label{cor:mills-original-notation}
Fix \(x>0\).  If
\begin{equation}\label{eq:mills-original-ordering}
        -1<p_1<\cdots<p_m,
        \qquad
        -1<q_1<\cdots<q_m,
\end{equation}
then
\begin{equation}\label{eq:mills-original-all-minor}
        \det\left[
        \frac{R_{(p_i+q_j)/2}(x)}
        {\Gamma\!\left(\frac{p_i+q_j+2}{4}\right)}
        \right]_{i,j=1}^m>0 .
\end{equation}
In particular, for \(p\neq q\),
\begin{equation}\label{eq:mills-original-logconvex}
        \left(
        \frac{R_{(p+q)/2}(x)}
        {\Gamma\!\left(\frac{p+q+2}{4}\right)}
        \right)^2
        <
        \frac{R_p(x)}{\Gamma((p+1)/2)}
        \frac{R_q(x)}{\Gamma((q+1)/2)} .
\end{equation}
\end{corollary}

\begin{proof}
Apply \cref{thm:mills-ratio} with
\(a_i=(p_i-1)/2\) and \(b_j=(q_j-1)/2\).
\end{proof}

\section{Tricomi rays and the Coulomb regularization}

The third application concerns Tricomi rays and the one-dimensional Coulomb
regularization.  This material is close to existing determinant and Tur\'an
literature, so we state two precise determinant consequences.  First, the ray
\(U(A,A+\delta,z)\) has a full Hankel total-positivity structure in the range
\(\delta<1\).  Second, the Coulomb regularization inherits the
\(\delta=1/2\) case.  The scalar \(2\times2\) consequence overlaps with known
shifted-parameter log-convexity results, and the theorem below supplies the
all-minor \(\STP_\infty\) strengthening.

For \(a>0\), \(z>0\), and real \(c\), Tricomi's confluent hypergeometric function has the integral representation
\begin{equation}\label{eq:U-integral}
        U(a,c,z)=\frac{1}{\Gamma(a)}\int_0^\infty e^{-zt}t^{a-1}(1+t)^{c-a-1}\,\dd t.
\end{equation}
This is Tricomi's standard integral representation \cite[13.4.4]{DLMF}; see also the determinant and Tur\'an literature for special functions and Tricomi functions \cite{IsmailLaforgia2007,BariczIsmail2013,KarpSitnik2010,Yang2025Extension}.
The scalar and \(2\times2\) Tricomi literature is substantial.  Ismail and
Laforgia studied determinant functions involving special functions
\cite{IsmailLaforgia2007}.  Baricz and Ismail proved sharp Tur\'an inequalities
and complete-monotonicity results for Tricomi Tur\'anians
\cite{BariczIsmail2013}.  Yang's shifted-parameter results imply scalar
log-convexity consequences for shifted Tricomi functions
\cite{Yang2025Extension}.  The theorem below gives a Hankel \(\STP_\infty\) upgrade along the special
ray \(c=a+\delta\).

We restrict to the ray
\begin{equation}\label{eq:tricomi-ray}
        c=a+\delta.
\end{equation}

\begin{theorem}[Gamma-weighted Tricomi rays]\label{thm:tricomi-gamma}
Fix \(z>0\) and \(\delta\in\R\).  The kernel
\begin{equation}\label{eq:gamma-weighted-U-kernel}
        (a,b)\longmapsto \Gamma(a+b)U(a+b,a+b+\delta,z)
\end{equation}
 is \(\STP_\infty\) on \((0,\infty)^2\).
\end{theorem}

\begin{proof}
For \(A=a+b>0\), \eqref{eq:U-integral} gives
\begin{equation}\label{eq:gammaU-moment}
        \Gamma(A)U(A,A+\delta,z)
        =\int_0^\infty t^{A-1}e^{-zt}(1+t)^{\delta-1}\,\dd t.
\end{equation}
This is the Mellin moment \(\int t^A\,\dd\mu(t)\) of the positive measure
\begin{equation}
        \dd\mu(t)=t^{-1}e^{-zt}(1+t)^{\delta-1}\,\dd t.
\end{equation}
The measure has infinite support and all moments with \(A>0\) are finite.  \Cref{thm:mellin} gives \(\STP_\infty\).  The condition \(a,b>0\) is used here only to ensure \(A=a+b>0\), the integrability condition at the origin.
\end{proof}

For \(\delta<1\), the gamma factor can be removed.  This is exactly the region needed for the Coulomb regularization.

\begin{theorem}[Unnormalized Tricomi ray total positivity]\label{thm:tricomi-unnormalized}
Fix \(z>0\) and \(\delta<1\).  Then
\begin{equation}\label{eq:U-ray-kernel}
        (a,b)\longmapsto U(a+b,a+b+\delta,z)
\end{equation}
 is \(\STP_\infty\) on \((0,\infty)^2\).  When \(\delta=1\), the same kernel is rank one, namely \(U(A,A+1,z)=z^{-A}\), so strict total positivity fails beyond order one.
\end{theorem}

\begin{proof}
Set \(A=a+b\), and in \eqref{eq:U-integral} use \(y=zt\):
\begin{equation}\label{eq:U-gamma-average}
        U(A,A+\delta,z)
        =z^{-A}\frac{1}{\Gamma(A)}\int_0^\infty e^{-y}y^{A-1}\left(1+\frac{y}{z}\right)^{\delta-1}\,\dd y.
\end{equation}
Let \(\alpha=1-\delta>0\).  Then
\begin{equation}\label{eq:tricomi-cm-multiplier}
        \left(1+\frac{y}{z}\right)^{-\alpha}
        =\frac{z^\alpha}{\Gamma(\alpha)}\int_0^\infty s^{\alpha-1}e^{-zs}e^{-ys}\,\dd s,
\end{equation}
so the multiplier in \eqref{eq:U-gamma-average} is completely monotone with non-atomic representing measure of infinite support.  \Cref{thm:gamma-cm} gives \(\STP_\infty\) for the gamma average, and the factor \(z^{-(a+b)}=z^{-a}z^{-b}\) is a positive row-column factor.  If \(\delta=1\), then the multiplier is \(1\) and \(U(A,A+1,z)=z^{-A}\), which has all minors of order \(\geq2\) equal to zero.
\end{proof}

\begin{remark}[Comparison with shifted-Tricomi log-convexity]
Taking \(A=a+b\), \cref{thm:tricomi-unnormalized} implies ordinary
log-convexity of \(A\mapsto U(A,A+\delta,z)\) for \(\delta<1\).  This scalar
statement lies in the range of shifted-parameter Tricomi log-convexity results,
for example Yang's results for shifted parameters \cite{Yang2025Extension}.
The determinantal statement proved here is that every Hankel minor
\begin{equation}\label{eq:tricomi-all-minor-comparison}
        \det\left[U(a_i+b_j,a_i+b_j+\delta,z)\right]_{i,j=1}^m
\end{equation}
is strictly positive for increasing \(a_i,b_j>0\).
\end{remark}

\begin{remark}[The ray boundary]
The gamma-weighted kernel in \cref{thm:tricomi-gamma} is strictly totally positive for every real \(\delta\).  The unnormalized theorem above removes the gamma factor by writing \((1+y/z)^{\delta-1}\) as a completely monotone multiplier, which is available in the range \(\delta<1\).  At \(\delta=1\) the ray degenerates to the rank-one kernel \(z^{-(a+b)}\).  The unnormalized ray for \(\delta>1\) requires separate arguments.
\end{remark}

\subsection{The Coulomb regularization as the ray \texorpdfstring{\(\delta=1/2\)}{delta = 1/2}}

The one-dimensional Coulomb regularization introduced by Ruskai--Werner and studied further by Alzer \cite{RuskaiWerner2000,Alzer2005} is
\begin{equation}\label{eq:Vq-def}
        V_q(x)=\frac{2e^{x^2}}{\Gamma(q+1)}\int_x^\infty e^{-t^2}(t^2-x^2)^q\,\dd t,
        \qquad q>-1,\quad x>0.
\end{equation}
Baricz and Pog\'any recorded, citing Baricz's thesis, the open problem whether
\(q\mapsto V_q(x)\) is log-convex on \((-1,\infty)\) for fixed \(x>0\);
see \cite[p.~87]{Baricz2008thesis} and \cite[p.~62]{BariczPogany2013}.
The theorem below gives a stronger statement: the whole Hankel kernel
\begin{equation}\label{eq:coulomb-hankel-kernel-display}
        (a,b)\mapsto V_{a+b-1}(x)
\end{equation}
is \(\STP_\infty\).  Its \(2\times2\) minor gives the requested scalar
log-convexity.  At scalar level, the identity
\begin{equation}\label{eq:coulomb-tricomi-overlap}
        V_q(x)=x^{2q+1}U(q+1,q+3/2,x^2)
\end{equation}
connects the consequence with Yang's shifted-Tricomi results
\cite{Yang2025Extension}; the theorem records the all-order Hankel
determinant hierarchy.

\begin{theorem}[Coulomb regularization determinant theorem]\label{thm:coulomb}
For every fixed \(x>0\), the kernel
\begin{equation}\label{eq:coulomb-kernel}
        (a,b)\longmapsto V_{a+b-1}(x),\qquad a,b>0,
\end{equation}
 is \(\STP_\infty\).  Consequently \(q\mapsto V_q(x)\) is strictly log-convex on \((-1,\infty)\).
\end{theorem}

\begin{proof}
The change of variables \(u=t^2-x^2\) in \eqref{eq:Vq-def} gives
\begin{equation}\label{eq:Vq-gamma-average}
        V_q(x)
        =
        \frac{1}{\Gamma(q+1)}
        \int_0^\infty u^q e^{-u}(u+x^2)^{-1/2}\,\dd u,
        \qquad q>-1 .
\end{equation}
Set \(A=q+1\) and
\begin{equation}\label{eq:coulomb-gx-def}
        g_x(u)=(u+x^2)^{-1/2}.
\end{equation}
For \(x>0\), \(g_x\) is completely monotone, since
\begin{equation}\label{eq:coulomb-cm-representation}
        (u+x^2)^{-1/2}
        =
        \frac{1}{\sqrt{\pi}}
        \int_0^\infty s^{-1/2}e^{-x^2s}e^{-us}\,\dd s .
\end{equation}
Thus
\begin{equation}\label{eq:coulomb-gamma-normalized-transform}
        V_{A-1}(x)
        =
        \frac{1}{\Gamma(A)}
        \int_0^\infty u^{A-1}e^{-u}g_x(u)\,\dd u
\end{equation}
is exactly the gamma-normalized transform in \cref{thm:gamma-cm}.  The
representing measure
\begin{equation}\label{eq:coulomb-representing-measure}
        \frac{1}{\sqrt{\pi}}s^{-1/2}e^{-x^2s}\,\dd s
\end{equation}
has infinite support on \((0,\infty)\).  Hence
\begin{equation}\label{eq:coulomb-stp-kernel-proof}
        (a,b)\longmapsto V_{a+b-1}(x)
\end{equation}
is \(\STP_\infty\) on \(a,b>0\).
\end{proof}

\begin{remark}[Tricomi identification]
Equivalently, the further change \(u=x^2t\) yields
\begin{equation}\label{eq:Vq-Tricomi}
        V_q(x)=x^{2q+1}U(q+1,q+3/2,x^2).
\end{equation}
Thus the Coulomb theorem is the \(\delta=1/2\) case of
\cref{thm:tricomi-unnormalized}, up to the positive row-column factor
\begin{equation}\label{eq:V-tricomi-row-column}
        x^{2(a+b)-1}=x^{-1}x^{2a}x^{2b}.
\end{equation}
\end{remark}

\begin{corollary}[Coulomb Tur\'an determinants]\label{cor:coulomb-dets}
Fix \(x>0\).  If \(-1<p_1<\cdots<p_m\) and \(-1<q_1<\cdots<q_m\), then
\begin{equation}\label{eq:coulomb-higher-dets}
        \det\left[V_{(p_i+q_j)/2}(x)\right]_{i,j=1}^m>0.
\end{equation}
In particular, for \(p\neq q\),
\begin{equation}\label{eq:coulomb-logconvex}
        V_{(p+q)/2}(x)^2<V_p(x)V_q(x).
\end{equation}
\end{corollary}

\begin{proof}
Apply \cref{thm:coulomb} with \(a_i=(p_i+1)/2\) and \(b_j=(q_j+1)/2\).
\end{proof}

\section{Consequences and limitations}

Two points guide the interpretation of the determinant results.  First, the
general moment-to-Hankel implication is classical; the special-function content
lies in recognizing the correct normalized order variable or the correct
positive atomic Mellin measure.  Second, scalar
log-convexity can enter through more than one route.  In the Coulomb case,
shifted-Tricomi log-convexity also yields the scalar consequence, while the
present theorem proves strict positivity of every minor in the Hankel matrix
of the order parameter.

For a positive Hankel kernel \(K(a,b)=F(a+b)\), strict log-convexity of \(F\) is the \(2\times2\) shadow of strict total positivity.  If \(u<v\), then
\begin{equation}
        \det\begin{pmatrix}
        F(u)&F((u+v)/2)\\
        F((u+v)/2)&F(v)
        \end{pmatrix}>0
\end{equation}
is exactly
\begin{equation}
        F((u+v)/2)^2<F(u)F(v).
\end{equation}
The determinant theorems above are therefore stronger than scalar Tur\'an or log-convexity inequalities: they give all Hankel minors on the same order domain.

The method has a structural boundary.  Positive row-column factors, increasing reparametrizations, positive push-forwards, and gamma-normalized completely monotone averages preserve the moment structure.  General quotients require additional comparison principles.  For example,
\begin{equation}
        \Gamma(a,x)=\int_x^\infty t^{a-1}e^{-t}\,\dd t
\end{equation}
is a positive Mellin moment in \(a\), so \((a,b)\mapsto \Gamma(a+b,x)\) is covered by \cref{thm:mellin}.  The normalized tail
\begin{equation}
Q(a,x)=\frac{\Gamma(a,x)}{\Gamma(a)}
\end{equation}
is a quotient of two positive Hankel moment kernels, and the denominator falls outside the row-column factor class.  Its determinant signs therefore require an additional quotient or comparison principle beyond order-moment transport alone.

The reciprocal Jackson \(q\)-gamma problem has the same warning sign.  The
positive kernel \(\Gamma_q(a+b)\) is covered by \cref{thm:qgamma}; the reciprocal
kernel \(1/\Gamma_q(a+b)\) is a quotient-type object, and its expected
sign-regularity requires a different principle.

Thus gamma normalization can expose hidden moment structure in order-parameter
inequalities, and atomic Mellin support can prove strictness for every minor
without any continuity of the representing measure.  In the Jackson
\(q\)-gamma case this proves the positive Hankel conjecture.  In the
Mills-ratio family it answers Yang's half-gamma log-convexity question and
gives all Hankel Tur\'an determinants.  In the Coulomb family it gives strict
total positivity of \((a,b)\mapsto V_{a+b-1}(x)\), whose first minor is the
scalar log-convexity inequality and whose higher minors extend beyond scalar
convexity.

\section*{Declaration of competing interest}
The author declares that there are no known competing financial interests or personal relationships that could have appeared to influence the work reported in this paper.

\section*{Data availability}
No data were used for the research described in this article.

\section*{Declaration of generative AI and AI-assisted technologies in the manuscript preparation process}
During the preparation of this work, the author used the Pudim AI research workflow, including Codex with GPT-5.5, to support manuscript organization, literature checking, revision review, and screening of the Jackson \(q\)-gamma application labelled APP-0080. After using these tools, the author reviewed and edited the content as needed and takes full responsibility for the content of the published article.

\end{document}